\documentclass[12pt]{amsart}

 \newtheorem{thm}{Theorem}[section]
 \newtheorem{cor}[thm]{Corollary}
 \newtheorem{lem}[thm]{Lemma}
 \newtheorem{prop}[thm]{Proposition}
 
 \theoremstyle{definition}
 
 \theoremstyle{remark}
 
 \numberwithin{equation}{section}

 \newcommand{\Real}{\mathbb{R}}

 \newcommand{\MTW}{\text{MTW}}
 \newcommand{\Hess}{\mathrm{Hess}}

 \newcommand{\Athreew}{{\bf (A3w)}\,}

\begin{document}

\title[Computable Necessary Conditions for Smooth Optimal Map]
{New Computable Necessary Conditions for the Regularity Theory of Optimal Transportation}

\author{Paul W.Y. Lee}
\email{plee@math.berkeley.edu}
\address{Department of Mathematics, University of California at Berkeley, 970 Evans Hall \#3840 Berkeley, CA 94720-3840 USA}

\thanks{The author was supported by the NSERC postdoctoral fellowship.}

\date{\today}

\begin{abstract}
We give new computable necessary conditions for a class of optimal transportation problems to have smooth solutions. 
\end{abstract}

\maketitle

\section{Introduction}

Let $\mu$ and $\nu$ be two Borel probability measures on the manifold $M$ and let $c:M\times M\longrightarrow\Real$ be a cost function. In the optimal transportation problem, one looks for a Borel map which minimizes the following total cost among all Borel maps $\varphi:M\longrightarrow M$ which push $\mu$ forward to $\nu$:
\[
\int_Mc(x,\varphi(x))d\mu(x)
\]
Here the push forward $\varphi_{*}\mu$ of a measure $\mu$ by a Borel map $\varphi$ is the measure defined by $\varphi_{*}\mu(U)=\mu(\varphi^{-1}(U))$ for all Borel sets $U\subseteq M$. 

Under some mild assumptions on the cost $c$ and the measures $\mu,\nu$, the above problem has a unique solution \cite{Br,Mc,BeBu,FaFi,AgLe,FiRi}. This unique solution is called the optimal map. There are various recent breakthroughs in understanding the regularity of the optimal map \cite{MaTrWa,TrWa,Lo1,Lo2,KiMc1}. The most important one is the introduction of a geometric object called the Ma-Trudinger-Wang (MTW) curvature. After the work of \cite{MaTrWa,TrWa,Lo1}, it is clear that certain non-negativity condition on the MTW curvature, called MTW condition, is necessary for the regularity theory of optimal maps. However, if the cost is not given by an explicit formula, then it is very hard to compute the MTW curvature and the MTW condition. When the cost $c$ is given by square of a Riemannian distance, the following is the only known computable condition which is necessary for the MTW condition. 

\begin{thm}\cite{Lo1}\label{Loeper}
Let $d$ be a Riemannian distance function on the manifold $M$ and assume that the cost $c$ is given by $c=d^2$. Then the MTW curvature satisfies 
\[
\MTW(u,0,w)=K(u,w)(|u|^2|w|^2-\left<u,w\right>^2)
\]
where $K(u,w)$ is the sectional curvature of the plane spanned by $u$ and $w$.  

In particular if the cost $c$ satisfies the weak MTW condition \Athreew, then the sectional curvature is non-negative. 
\end{thm}

The purpose of this paper is twofold. First, we consider cost functions arising from natural mechanical systems. More precisely, let $\left<\cdot,\cdot\right>$ be a Riemannian metric on the manifold $M$ and let $|\cdot|$ be the corresponding norm. Let $V:M\longrightarrow \Real$ be a smooth function on the manifold $M$, called the potential, and let $L:TM\to\Real$ be the Lagrangian defined by $L(x,v)=\frac{1}{2}|v|^2-V(x)$. The cost functions that we are interested in are given by 
\begin{equation}\label{cost}
c(x,y)=\inf\int_0^1L(\gamma(t),\dot\gamma(t))dt,
\end{equation}
where the infimum is taken over all smooth curves $\gamma(\cdot)$ satisfying $\gamma(0)=x$ and $\gamma(1)=y$. 

In the first part of the paper, we give computable necessary conditions for the cost defined in (\ref{cost}) to satisfy the weak MTW condition \Athreew (see Theorem \ref{main1}). The following is a simple corollary of Theorem \ref{main1}.  
\begin{thm}\label{main1-1}
Let $x$ be a maximum point of the potential $V$. Assume that $V$ also satisfy the following
\[
\nabla V_x=0 \quad \text{and}\quad \Hess V_x=0.
\]
Let $u$ and $w$ be two tangent vectors based at $x$. Then the MTW curvature for the cost $c$ defined in (\ref{cost}) satisfies 
\[
\MTW(u,0,w)=\left<R(w,u)w,u\right>+\frac{1}{20}\left<\nabla^2_{w}\nabla_{u}\nabla V_x,u\right>.
\]

In particular if the cost $c$ satisfies the weak MTW condition \Athreew, then the following holds
\[
\left<R(w,u)w,u\right>+\frac{1}{20}\left<\nabla^2_{w}\nabla_{u}\nabla V_x,u\right>\geq 0
\] 
for all orthogonal pairs $(u,w)$ of tangent vectors $\left<u,w\right>=0$. 
\end{thm}

We remark that the condition $\Hess V_x=0$ is not completely necessary and it can be replaced by a more complicated condition (see Theorem \ref{main1}). Note that when the potential $V\equiv 0$, the cost $c$ is given by the square of the corresponding Riemannian distance $d$ and Theorem \ref{main1-1} reduces to Theorem \ref{Loeper}. As a corollary of Theorem \ref{main1-1}, we have the following. 

\begin{cor}
Let $A$ be a $n\times n$ matrix satisfying 
\[
(\left<Au,w\right>+\left<Au,w\right>)^2+2\left<Au,u\right>\left<Aw,w\right>>0
\]
for a pair of vectors $(u,w)$ in $\Real^n$ which are orthogonal $\left<u,w\right>=0$. 

Let $V:\Real^n\to\Real$ be a potential satisfying
\[
V(x)=-\left<Ax,x\right>^2+O(|x|^5)\quad \text{as} \quad |x|\to 0
\]
and let $L$ be the Lagrangian defined by $L(x,v)=\frac{1}{2}|v|^2-V(x)$ where $|\cdot|$ is the Euclidean norm. 

Then the MTW curvature for the cost $c$ defined in (\ref{cost}) does not satisfy the weak MTW condition \Athreew. 
\end{cor}

In the second part of the paper, we focus on the case $c=d^2$, where $d$ is a Riemannian distance on the manifold $M$. We go beyond Theorem \ref{Loeper} and consider higher order necessary conditions for the MTW conditions. More precisely, according to Theorem \ref{Loeper}, the Riemannian manifold $M$ necessarily has non-negative sectional curvature if the cost $c=d^2$ satisfies the MTW conditions. However, when the sectional curvature is only non-negative, Theorem \ref{Loeper} does not tell us anything about the MTW conditions near where the sectional curvature vanishes. To understand the MTW conditions near these points, we consider the higher order Taylor expansion of the MTW curvature in the $v$-variable. If we assume that the sectional curvature $K(u,w)$ of the plane spanned by $u$ and $w$ vanishes, then the zeroth order term in the Taylor expansion of the MTW curvature $\MTW(u,v,w)$ in $v$ at the origin vanishes by Theorem \ref{Loeper}. Therefore, if the MTW curvature satisfies $\MTW(u,v,w)\geq 0$ for all small enough $v$, then necessarily the first order term in the Taylor expansion vanishes and the second order term is non-negative. As a result, we get new necessary conditions for the cost $d^2$ to satisfy the weak MTW condition (Theorem \ref{main2}). When the manifold is two-dimensional, the conditions are simplified and give the following simple result. 

\begin{thm}\label{main2cor1}
Assume that $M$ is a two dimensional Riemannian manifold with Riemannian distance function $d$. If the cost $c=d^2$ satisfies the weak MTW condition \Athreew, then $M$ has non-negative Gauss curvature and the Riemannian curvature $R$  satisfies 
\[
3\left<(\nabla_u\nabla_w R)(w,u)w,u\right>^2\leq
2\left<(\nabla_w^2R)(w,u)w,u\right>\left<(\nabla_u^2R)(w,u)w,u\right>
\]
for each pair $(u,w)$ of orthogonal vectors $\left<u,w\right>=0$ which spanned a plane with zero sectional curvature (i.e. $\left<R(u,w)u,w\right>=0$).
\end{thm}

As an example, we consider the two dimensional Euclidean space $\Real^2$ equipped with the metric 
\begin{equation}\label{conformal}
\left<u,v\right>=e^{2f(x)}u\cdot v,
\end{equation}
where $u\cdot w$ denotes the usual dot product and $f(x,y)=x^3y+ax^2y^2+xy^3+a_4y^4$. 

For these Riemannian metrics, the Gauss curvature is zero at the origin and nonzero everywhere else if 
$a\leq -3$. As a result of Theorem \ref{main2cor1}, we get the following. 

\begin{thm}\label{main3}
Assume that the Riemannian distance $d$ is defined by the Riemannian metric $\left<\cdot,\cdot\right>$ given in (\ref{conformal}). If the cost $c=d^2$ satisfies the weak MTW condition \Athreew, then 
\[
a\leq-\sqrt{\frac{27}{2}}. 
\]
\end{thm}

\

\section{Background:  The MTW Curvature}\label{background}

In this section, we will review some basic facts about the optimal transportation problem and the definition of the Ma-Trudinger-Wang (MTW) curvature. The assumptions in the theorems stated in this section are simplified to avoid heavy notation. The corresponding theorems with relaxed assumptions can be found, for instance, in \cite{Vi}.

Let $\left<\cdot,\cdot\right>$ be a Riemannian metric on a manifold $M$ and let $V:M\to\Real$ be a smooth function which is bounded above. Let $L:TM\to\Real$ be the Lagrangian defined by 
\[
L(x,v)=\frac{1}{2}|v|^2-V(x). 
\]

In this paper, we are mainly interested in the cost $c$ defined by 
\begin{equation}\label{costagain}
c(x,y)=\inf\int_0^1L(\gamma(t),\dot\gamma(t))dt,
\end{equation}
where the infimum is taken over all smooth curves $\gamma(\cdot)$ satisfying $\gamma(0)=x$ and $\gamma(1)=y$. 

Curves $t\mapsto\gamma(t)$ which achieve the above infimum are called curves of least action and they satisfy the following equation (see \cite{LeMc})
\begin{equation}\label{least}
\partial_t^2\gamma=-\nabla V_\gamma. 
\end{equation}
Here we abuse notation and denote the covariant derivative by $\partial_t$. The same convention will be used throughout this paper. 

If $t\mapsto\gamma(t)$ is a curve of least action with initial velocity $v$, then the $c$-exponential map $\exp^c$ is defined by 
\[
\exp^c(v)=\gamma(1). 
\]
Note that, unlike the Riemannian case, $t\mapsto\exp^c(tv)$ is not a curve of least action in general.

Let $\mu$ and $\nu$ be two Borel probability measures with compact supports on the manifold $M$. We recall that the optimal transportation problem is the following minimization problem: 

Find a Borel map which minimizes the following total cost among all Borel maps $\varphi:M\longrightarrow M$ which push $\mu$ forward to $\nu$:
\[
\int_Mc(x,\varphi(x))d\mu(x)
\]
Here the push forward $\varphi_{*}\mu$ of a measure $\mu$ by a Borel map $\varphi$ is the measure defined by $\varphi_{*}\mu(U)=\mu(\varphi^{-1}(U))$ for all Borel sets $U\subseteq M$. 

\begin{thm}\label{MongeExist}
Suppose that the cost $c$ is given by (\ref{costagain}) and the measure $\mu$ is absolutely continuous with respect to the Lebesgue measure. Then there is a solution $\varphi$ (called the optimal map) to the above optimal transportation problem which is unique $\mu$-almost everywhere. Moreover, there exists a Lipschitz function $f:M\to\Real$ such that the unique optimal map $\varphi$ is given by
\[
\varphi(x)=\exp^c(\nabla f(x)). 
\]
\end{thm}

Next, we discuss the main object of this paper, the Ma-Trudinger-Wang (MTW) curvature. Let $u$, $v$, and $w$ be vectors based at the point $x$. The MTW curvature $\MTW$ is defined by 
\[
\MTW(u,v,w)=-\frac{3}{2}\partial_t^2\partial_s^2c(\sigma(t),\exp^c(v+sw))\Big|_{s=t=0},
\]
where $\sigma$ is any curve with initial velocity $u$ (i.e. $\partial_t\sigma\Big|_{t=0}=u$).

Finally, we can state the MTW conditions. Let $\mathcal O$ be the set of all pairs of points $(x,y)$ contained in the product $M\times M$ such that 
\begin{enumerate}
\item there exists a unique curve of least action $\gamma$ satisfying $\gamma(0)=x$ and $\gamma(1)=y$, 
\item the map $d\exp^c\Big|_{T_xM}$ is a submersion at $\dot\gamma(0)$. 
\end{enumerate}
It is known that the cost function $c$ is smooth on the set $\mathcal O$ (see, for instance, \cite{LeMc}) and the MTW curvature is well-defined. Let $\tilde {\mathcal O}$ be the subset of all initial velocities $\partial_t\gamma\Big|_{t=0}$ in the definition of $\mathcal O$. Then the weak MTW condition is given by the following:

The cost $c$ satisfies the weak MTW condition \Athreew on a subset $\mathcal M$ of $\mathcal O$ if $$\MTW(u,v,w)\geq 0$$ on the set $$\{(u,v,w)|v\in\tilde{\mathcal O},(x,\exp^c(v))\in\mathcal M,\left<u,w\right>=0\},$$

The relevance of these conditions to the regularity theory of optimal maps can be found
in \cite{MaTrWa,Lo1,Lo2,KiMc1,TrWa,LoVi,FiLo,FiRi,FiRiVi1,FiRiVi2,FiKiMc2}. 

\

\section{The Ma-Trudinger-Wang curvature and the Riemannian curvature}

In this section, we give a formula for the MTW curvature in terms of the change in the Riemannian curvature and the Hessian of the potential along curves of least action. Before stating the precise result, let us introduce the following notations. Let $u$, $v$, and $w$ be tangent vectors based at the point $x$ and let $\tau\mapsto\gamma_s(\tau)$ be the curve of least action with initial velocity $v+sw$ (i.e. $\partial_\tau\gamma_s\Big|_{\tau=0}=v+sw$). Let $\tau\mapsto U_s(\tau)$ be the parallel translation of the vector $u$ along the curve $\tau\mapsto\gamma_s(\tau)$. Let $\tau\mapsto J_s(\tau)$ be a vector field defined along the curve $\tau\mapsto\gamma_s(\tau)$, called Jacobi field. It is defined as the solution of the following Jacobi equation 
\[
\partial_\tau^2J+R(\partial_\tau\gamma,J)\partial_\tau\gamma+\Hess V_\gamma(J)=0. 
\]
We assume that the Jacobi field $J(\cdot)$ also satisfies the following boundary conditions $J_s(0)=u$, $J_s(1)=0$, and $J_s(\tau)\neq 0$ for all time $\tau$ in the interval $(0,1)$. 

\begin{thm}\label{curvature-curvature}
The MTW curvature is given by 
\begin{equation}\label{curvature-0}
\begin{split}
&\MTW(u,v,w)\\
&=\frac{3}{2}\int_0^1\int_0^{\bar\tau}\partial_s^2\left<R(\partial_\tau\gamma,J)\partial_\tau\gamma+\Hess V_\gamma(J),U\right>d\tau d\bar\tau\Big|_{s=0}.  
\end{split}
\end{equation}
\end{thm}

\begin{proof}
The Jacobi field $J$ satisfies the following Jacobi equation 
\[
\partial_\tau^2J+R(\partial_\tau\gamma,J)\partial_\tau\gamma+\Hess V_\gamma(J)=0. 
\]

It follows that 
\[
\partial_\tau\left<\partial_\tau J,U\right>+\left<R(\partial_\tau\gamma,J)\partial_\tau\gamma,U\right>+\left<\Hess V_\gamma(J),U\right>=0. 
\]

If we integrate with respect to the variable $\tau$, then the above equation becomes 
\[
\begin{split}
\left<U,\partial_\tau J\right>\Big|_{\tau=0}
&=\left<U,\partial_\tau J\right>\Big|_{\tau=\bar\tau} +\int_0^{\bar\tau}\left<R(\partial_\tau\gamma,J)\partial_\tau\gamma+\Hess V_\gamma(J),U\right>d\tau\\ 
&=\partial_\tau\left<U, J\right>\Big|_{\tau=\bar\tau} +\int_0^{\bar\tau}\left<R(\partial_\tau\gamma,J)\partial_\tau\gamma+\Hess V_\gamma(J),U\right>d\tau. 
\end{split}
\]

Now if we integrate again with respect to $\bar\tau$ and use the boundary conditions for $J$, then we have 
\begin{equation}\label{curvature-1}
\begin{split}
&\left<U,\partial_\tau J\right>\Big|_{\tau=0}\\ 
&=-\left<u,u\right> +\int_0^1\int_0^{\bar\tau}\left<R(\partial_\tau\gamma,J)\partial_\tau\gamma+\Hess V_\gamma(J),U\right>d\tau d\bar\tau. 
\end{split}
\end{equation}

By \cite[Theorem 3.1]{LeMc}, we know that the MTW curvature is given by 
\[
\MTW(u,v,w)=\frac{3}{2}\partial_s^2\left<U,\partial_\tau J\right>\Big|_{\tau=s=0}.  
\]

The result follows from this and (\ref{curvature-1}). 
\end{proof}

\

\section{Zeroth Order Condition for Natural Mechanical Actions}

In this section, we give the proof of the following main theorem. 

\begin{thm}\label{main1}
Let $x$ be a maximum point of the potential $V$. Let $u$ and $w$ be two tangent vectors based at $x$. Then the MTW curvature for the cost $c$ defined in (\ref{costagain}) satisfies 
\[
\begin{split}
&\MTW(u,0,w) 
=\frac{3}{2}\Bigg(\int_0^1\int_0^{\bar\tau}2\left<R(\partial_\tau\bar w,\tilde u)\partial_\tau\bar w,u\right>\\
&+\left<\Hess V_x(\tilde u),\int_0^\tau R(\partial_\tau\bar w,\bar w)ud\tau\right>
+\left<\nabla^2_{\bar w}\nabla_{\tilde u}\nabla V_x,u\right>d\tau d\bar\tau\Bigg)\Bigg|_{s=0},
\end{split}
\]
where $\tilde u$ and $\bar w$ satisfies the following linear ordinary differential equation
\[
\partial_\tau^2 u=-\Hess V_x(u), 
\]
$\bar w$ satisfies the initial conditions 
\[
\bar w\Big|_{\tau=0}=0\quad \partial_\tau\bar w\Big|_{\tau=0}=w,
\] 
and $\tilde u$ satisfies the boundary conditions
\[
\tilde u\Big|_{\tau=0}=u,\quad \tilde u\Big|_{\tau=1}=0, \quad \tilde u\neq 0 \quad \text{if } 0<\tau<1. 
\]
\end{thm}

Let us first give the proof of Theorem \ref{main1-1}. 

\begin{proof}[Proof of Theorem \ref{main1-1}]
Since $\Hess_xV=0$, we have $\bar w=\tau w$ and $\tilde u=(1-\tau)u$. If we substitute this back into the formula for $\MTW(u,0,w)$ in Theorem \ref{main1}, then we have 
\[
\begin{split}
&\MTW(u,0,w)\\
&=\frac{3}{2}\int_0^1\int_0^{\bar\tau}2(1-\tau)\left<R(w,u)w,u\right>\\
&+\tau^2(1-\tau)\left<\nabla^2_{w}\nabla_{u}\nabla V_x,u\right>d\tau d\bar\tau\Big|_{s=0}\\
&=\left<R(w,u)w,u\right>+\frac{1}{20}\left<\nabla^2_{w}\nabla_{u}\nabla V_x,u\right>.
\end{split}
\]
\end{proof}

\begin{proof}[Proof of Theorem \ref{main1}] 
First let us note that the cost $c$ is smooth at the point $(x,x)$. Indeed, since $x$ is a maximum point of the potential $V$, the constant curve $\gamma(\cdot)\equiv x$ is the unique minimizer satisfying $\gamma(0)=x$ and $\gamma(1)=x$. Let $\tau\mapsto J(\tau)$ be a vector field defined along $\gamma$ which satisfies the Jacobi equation
\[
\partial_\tau^2J+R(\partial_\tau\gamma,J)\partial_\tau\gamma+\Hess V_\gamma(J)=0. 
\]
Since $\gamma\equiv x$ and $\partial_\tau\gamma=0$, it follows that 
\[
\partial_\tau^2J+\Hess V_x(J)=0. 
\]
The point $x$ is a maximum point of the potential $V$, so the Hessian of the potential $\Hess V$ is non-positive definite. Therefore, if $J$ satisfies the boundary conditions $J\Big|_{\tau=0}=J\Big|_{\tau=1}=0$, then $J\equiv 0$. It follows from \cite[Theorem 7.6]{LeMc} that the map $d\exp^c\Big|_{T_xM}$ has full rank at the origin. Therefore, the cost function $c$ is smooth at the point $(x,x)$ by \cite[Theorem 7.7]{LeMc} (see also Section \ref{background}). 

Let us first introduce some notations. Let $\tau\mapsto\gamma_{s,t}(\tau)$ be the curve of least action with initial velocity $tv+sw$. Let $\tau\mapsto U_{s,t}(\tau)$ be the parallel translation of the vector $u$ along the curve $\tau\mapsto \gamma_{s,t}(\tau)$. Let $\tau\mapsto J_{s,t}(\tau)$ be the Jacobi field defined along the curve $\tau\mapsto\gamma_{s,t}(\tau)$ which satisfies the conditions $J_{s,t}(0)=u$, $J_{s,t}(1)=0$, and $J_{s,t}(\tau)\neq 0$ for all $\tau$ in the interval $(0,1)$. 

If we expand the term $\partial_s^2\left<R(\partial_\tau\gamma,J)\partial_\tau\gamma,U\right>$ in (\ref{curvature-0}), then we have 
\begin{equation}\label{long-1}
\begin{split}
&\partial_s^2\left<R(\partial_\tau\gamma,J)\partial_\tau\gamma,U\right>\\
&=\left<(\partial_s^2R)(\partial_\tau\gamma,J)\partial_\tau\gamma,U\right> +\left<R(\partial_s^2\partial_\tau\gamma,J)\partial_\tau\gamma,U\right>\\
&+\left<R(\partial_\tau\gamma,\partial_s^2 J)\partial_\tau\gamma,U\right> +\left<R(\partial_\tau\gamma,J)\partial_s^2\partial_\tau\gamma,U\right>\\
&+\left<R(\partial_\tau\gamma,J)\partial_\tau\gamma,\partial_s^2U\right> 
+2\left<(\partial_s R)(\partial_s\partial_\tau\gamma,J)\partial_\tau\gamma,U\right>\\
&+2\left<(\partial_sR)(\partial_\tau\gamma,\partial_sJ)\partial_\tau\gamma,U\right> +2\left<(\partial_sR)(\partial_\tau\gamma,J)\partial_s\partial_\tau\gamma,U\right>\\
&+2\left<(\partial_sR)(\partial_\tau\gamma,J)\partial_\tau\gamma,\partial_sU\right> +2\left<R(\partial_s\partial_\tau\gamma,\partial_sJ)\partial_\tau\gamma,U\right>\\
&+2\left<R(\partial_s\partial_\tau\gamma,J)\partial_s\partial_\tau\gamma,U\right> +2\left<R(\partial_s\partial_\tau\gamma,J)\partial_\tau\gamma,\partial_sU\right>\\
&+2\left<R(\partial_\tau\gamma,\partial_sJ)\partial_s\partial_\tau\gamma,U\right> +2\left<R(\partial_\tau\gamma,\partial_sJ)\partial_\tau\gamma,\partial_sU\right>\\ &+2\left<R(\partial_\tau\gamma,J)\partial_s\partial_\tau\gamma,\partial_sU\right>.
\end{split}
\end{equation}

By (1) of Lemma \ref{longlem-1}, (\ref{long-1}) simplifies to 
\[
\partial_s^2\left<R(\partial_\tau\gamma,J)\partial_\tau\gamma,U\right>\Big|_{s=t=0} =2\left<R(\partial_s\partial_\tau\gamma,J)\partial_s\partial_\tau\gamma,U\right>\Big|_{s=t=0}.
\]

If we apply (2) of Lemma \ref{longlem-1}, (1) of Lemma \ref{longlem-2}, and Lemma \ref{longlem-3}, then the above becomes
\begin{equation}\label{short-1}
\partial_s^2\left<R(\partial_\tau\gamma,J)\partial_\tau\gamma,U\right>\Big|_{s=t=0} =2\left<R(\partial_\tau\bar w,\tilde u)\partial_\tau\bar w,u\right>. 
\end{equation}

If we expand the other term $\partial_s^2\left<\Hess V_\gamma(J),U\right>$ in (\ref{curvature-0}), then we have 
\begin{equation}\label{long-2}
\begin{split}
\partial_s^2\left<\Hess V_\gamma(J),U\right>
&=\left<\Hess V_\gamma(J),\partial_s^2U\right>\\
&+\left<\partial_s^2(\Hess V_\gamma(J)),U\right>\\& +2\left<\partial_s(\Hess V_\gamma(J)),\partial_sU\right>. 
\end{split}
\end{equation}

By (1) and (2) of Lemma \ref{longlem-2}, (\ref{long-2}) becomes 
\[
\begin{split}
\partial_s^2\left<\Hess V_\gamma(J),U\right>\Big|_{s=t=0} 
&=\left<\partial_s^2(\Hess V_\gamma(J)),U\right> \\
&+\left<\Hess V_x(\tilde u),\int_0^\tau R(\partial_\tau\bar w,\bar w)ud\tau\right>\Big|_{s=t=0}. 
\end{split}
\]

By (1) of Lemma \ref{longlem-3}, it follows that 
\[
\begin{split}
\partial_s^2\left<\Hess V_\gamma(J),U\right>\Big|_{s=t=0} 
&=\left<\nabla^2_{\bar w}\nabla_{\tilde u}\nabla V_\gamma,u\right>\\
&+\left<\Hess V_x(\tilde u),\int_0^\tau R(\partial_\tau\bar w,\bar w)ud\tau\right>. 
\end{split}
\]

Finally we combine this with (\ref{short-1}) and Theorem \ref{curvature-curvature} to finish the proof. 
\end{proof}

\

\section{Higher Order Conditions in the Riemannian Case}

In this section, we consider the first and the second order terms of the MTW curvature in the $v$-variable. More precisely, we will prove the following second main result of the paper. 

\begin{thm}\label{main2}
Let $d$ be a Riemannian distance function on the manifold $M$. Assume that the cost function $c$ is given by $c=d^2$ Then the MTW curvature satisfies  
\[
\begin{split}
&\partial_t\MTW(u,tv,w)\Big|_{t=0} \\
&=\frac{1}{2}\left<(\nabla_wR)(w,u)v,u\right>+\frac{1}{4}\left<(\nabla_vR)(w,u)w,u\right> 
\end{split}
\]

and 

\[
\begin{split}
&\partial_t^2\MTW(u,tv,w)\Big|_{t=0}\\
&=\frac{1}{10}\left<(\nabla_w^2R)(v,u)v,u\right> 
-\frac{1}{5}\left<R(v,u)u,R(v,w)w\right>\\
&+\frac{4}{15}\left<R(v,u)v,R(w,u)w\right> 
+\frac{2}{5}\left<(\nabla_v\nabla_w R)(w,u)v,u\right> \\
&+\frac{1}{10}\left<(\nabla_v^2R)(w,u)w,u\right> 
-\frac{1}{5}\left<R(w,u)u,R(v,w)v\right>\\
&+\frac{4}{15}\left(\left<R(w,u)v,R(w,u)v\right>+\left<R(v,u)w,R(w,u)v\right>\right)\\
&+\frac{1}{3}\left(\left<R(w,u)v,R(v,w)u\right>+\left<R(v,u)w,R(v,w)u\right>\right).
\end{split}
\]

\end{thm}

The proof of Theorem \ref{main2} will be postponed to Section \ref{proof}. For the rest of this section, we will state and prove the consequences of Theorem \ref{main2}. 

\begin{thm}\label{main2-2}
Let $d$ be a Riemannian distance of non-negative sectional curvature on the manifold $M$. Assume that the cost $c=d^2$ satisfies the condition \Athreew. Then for each pair $(u,w)$ of orthogonal vectors $\left<u,w\right>=0$ for which the plane spanned by $u$ and $w$ has zero sectional curvature, we have  
\begin{equation}\label{main2-2-1}
\left<(\nabla_w R)(w,u)v,u\right>=0 
\end{equation}
and 
\[
\begin{split}
G(u,v,w)&:=\frac{1}{10}\left<(\nabla_w^2R)(v,u)v,u\right> 
-\frac{1}{5}\left<R(v,u)u,R(v,w)w\right>\\
&+\frac{2}{5}\left<(\nabla_v\nabla_w R)(w,u)v,u\right> 
+\frac{1}{10}\left<(\nabla_v^2R)(w,u)w,u\right> \\
&+\frac{4}{15}\left(\left<R(w,u)v,R(w,u)v\right>
+\left<R(v,u)w,R(w,u)v\right>\right)\\
&+\frac{1}{3}\left(\left<R(w,u)v,R(v,w)u\right>
+\left<R(v,u)w,R(v,w)u\right>\right)\geq 0.
\end{split}
\]
\end{thm}

\begin{proof}
We extend the vectors $u,v,w$ to vector fields $U,V,W$, respectively. Moreover, we assume that $U,V,W$ are constant vector fields in a geodesic normal coordinate neighborhood of the point $x$. By assumption, the function $\left<R(W,U)W,U\right>$ has a minimum at $x$. It follows that 
\[
\left<(\nabla_vR)(w,u)w,u\right>=V\left<R(W,U)W,U\right>=0. 
\]
This proves the first equality. For the inequality involving $G$, we need to show that $R(u,w)u=R(w,u)w=0$. Indeed, we know that $R(u,\cdot)u$ is a symmetric operator. Since the manifold has non-negative sectional curvature, $R(u,\cdot)u$ is non-negative definite. We also have $\left<R(u,w)u,w\right>=0$, so it follows that $R(u,w)u=0$. A similar argument shows that $R(w,u)w=0$. 
\end{proof}

\begin{proof}[Proof of Theorem \ref{main2cor1}]
Since the manifold $M$ is two dimensional, $v=au+bw$ for some constants $a$ and $b$. It follows that (\ref{main2-2-1}) becomes 
\[
0=\left<(\nabla_w R)(w,u)v,u\right>=b\left<(\nabla_w R)(w,u)w,u\right>. 
\]

However, by the proof of Theorem \ref{main2-2}, $\left<(\nabla_w R)(w,u)w,u\right>=0$. Therefore, (\ref{main2-2-1}) is satisfied automatically. 

Since the Riemannian curvature $R$ satisfies $R(u,w)u=R(w,u)w=0$ and $M$ is 2-dimensional, the term $G$ is simplified to 
\[
\begin{split}
G(u,v,w)&=\frac{1}{10}\left<(\nabla_w^2R)(v,u)v,u\right> \\
&+\frac{2}{5}\left<(\nabla_v\nabla_w R)(w,u)v,u\right> 
+\frac{1}{10}\left<(\nabla_v^2R)(w,u)w,u\right> \\
&=\frac{3b^2}{5}\left<(\nabla_w^2R)(w,u)w,u\right> 
+\frac{ab}{10}\left<(\nabla_{w}\nabla_u R)(w,u)w,u\right>\\
&+\frac{ab}{2}\left<(\nabla_u\nabla_w R)(w,u)w,u\right>
+\frac{a^2}{10}\left<(\nabla_{u}^2R)(w,u)w,u\right> .
\end{split}
\]

Let us extend $u$ and $v$ to vector fields $U$ and $W$, respectively, which are constant in a geodesic normal coordinate neighborhood. Since the Gauss curvature at $x$ vanishes and the covariant derivatives satisfy  $\nabla_UV\Big|_x=\nabla_VU\Big|_x=\nabla_UU\Big|_x=\nabla_VV\Big|_x=0$, it follows that 
\[
\begin{split}
&\left<(\nabla_u\nabla_w R)(w,u)w,u\right>=\nabla_U\nabla_W\left<R(W,U)W,U\right>\Big|_x 
\\&=\nabla_W\nabla_U\left<R(W,U)W,U\right>\Big|_x =\left<(\nabla_w\nabla_u R)(w,u)w,u\right>. 
\end{split}
\]

Therefore, the formula for $G$ simplifies to 
\begin{equation}\label{Gsimplified}
\begin{split}
&G(u,v,w)=\frac{3b^2}{5}\left<(\nabla_w^2R)(w,u)w,u\right> \\
&+\frac{3ab}{5}\left<(\nabla_{w}\nabla_u R)(w,u)w,u\right>
+\frac{a^2}{10}\left<(\nabla_{u}^2R)(w,u)w,u\right> .
\end{split}
\end{equation}

Since the manifold $M$ has non-negative Gauss curvature and has zero Gauss curvature at $x$. We have 
\[
\left<(\nabla_w^2R)(w,u)w,u\right>=\nabla_W^2\left<R(W,U)W,U\right>\Big|_x\geq 0. 
\]

Therefore, if the quadratic in (\ref{Gsimplified}) satisfies $G(u,v,w)\geq 0$ for all $v$, then the discriminant is non-positive and it follows that 
\[
3\left<(\nabla_{w}\nabla_u R)(w,u)w,u\right>^2\leq 2\left<(\nabla_w^2R)(w,u)w,u\right> \left<(\nabla_{u}^2R)(w,u)w,u\right> .
\]

\end{proof}

\

\section{Example}

In this section, we discuss the proof of Theorem \ref{main3}. Recall that we consider the following Riemannian metric $\left<\cdot,\cdot\right>$ on $\Real^2$:
\begin{equation}\label{conformalagain}
\left<u,v\right>=e^{2f(x)}u\cdot v. 
\end{equation}

Let us denote the gradient and the Laplacian of the usual Euclidean metric by $\nabla$ and $\Delta$, respectively. Let $\tilde\nabla$ and $R$ be, respectively, the Levi-Civita connection and the Riemannian curvature of the Riemannian metric $\left<\cdot,\cdot\right>$. 

\begin{lem}\label{Gauss}
The Levi-Civita connection $\tilde\nabla$ and the Riemannian curvature $R$ of the metric $\left<\cdot,\cdot\right>$ are given by 
\[
\tilde\nabla_{\partial_x}\partial_y=\tilde\nabla_{\partial_y}\partial_x=f_x\partial_y+f_y\partial_x,\quad \tilde\nabla_{\partial_x}\partial_x=-\tilde\nabla_{\partial_y}\partial_y=f_x\partial_x-f_y\partial_y
\]
\[
K=-(\Delta f)\cdot e^{-2f} 
\]
where $K$ denotes the Gauss curvature with respect to the Riemannian metric $\left<\cdot,\cdot\right>$. 
\end{lem}

\begin{proof}
It follows immediately from, for instance, \cite[Theorem 1.159]{Be}. 
\end{proof}

\begin{prop}\label{GaussPos}
Let $f(x,y)=x^3y+ax^2y^2+xy^3$. Then the Riemannian metric defined by (\ref{conformalagain}) has non-negative Gauss curvature if 
\begin{equation}\label{egpositive}
a\leq -3. 
\end{equation}
\end{prop}

\begin{proof}
A computation shows that 
\[
\Delta f(x,y)=2ax^2+12xy+2ay^2. 
\]

It follows from Lemma \ref{Gauss} that the Gauss curvature is non-negative if and only if the quadratic $2ax^2+12xy+2ay^2$ is non-positive.  This, in turn, is equivalent to $a<0$ and $144-16a^2\leq 0$. 
\end{proof}

\begin{proof}[Proof of Theorem \ref{main3}]
Assume that the cost $c=d^2$ satisfies the weak MTW conditon. By Proposition \ref{GaussPos} and Theorem \ref{Loeper}, we have $a\leq -3$. Let $U$ and $W$ be two constant vector fields which are orthonormal with respect to the Euclidean metric. Assume that $U\Big|_{(0,0)}=u$ and $W\Big|_{(0,0)}=w$. Then
\begin{equation}\label{mainpf3-1}
\begin{split}
\left<R(U,W)U,W\right>&=e^{4f}K=\\
&=-e^{2f}(\Delta f)\\
&=-e^{2f}(2ax^2+12xy+2ay^2).
\end{split}  
\end{equation}

It follows from Lemma \ref{Gauss} that  
\[
\nabla_UU\Big|_{(0,0)}=\nabla_UW\Big|_{(0,0)}=\nabla_U^2U\Big|_{(0,0)}=\nabla_U^2W\Big|_{(0,0)}=0.
\] Therefore, 
\begin{equation}\label{mainpf3-2}
\left<\tilde\nabla_u^2R(u,w)u,w\right>=\tilde\nabla_U^2\left<R(U,W)U,W\right>\Big|_{(0,0)}. 
\end{equation}

Since $\tilde\nabla g=e^{-2f}\nabla g$ for each smooth function $g$, it follows from Lemma \ref{Gauss} again that 
\[
\left<\tilde\nabla_u\tilde\nabla g,u\right>=\left<\nabla_u\nabla g,u\right>. 
\]

It follows from this, (\ref{mainpf3-1}), and (\ref{mainpf3-2}) that 
\[
\left<\tilde\nabla_u^2R(u,w)u,w\right>=-4(au_1^2+6u_1u_2+au_2^2). 
\]

Similar calculations show that 
\[
\left<\tilde\nabla_w^2R(u,w)u,w\right>=-4(aw_1^2+6w_1w_2+aw_2^2) 
\]
and
\[
\left<\tilde\nabla_w\tilde\nabla_uR(u,w)u,w\right>=-4(au_1w_1+3u_1w_2+3w_1u_2+au_2w_2). 
\]

Since $u$ and $w$ are orthogonal, we can assume that $w$ is given by $w=-u_2\partial_x+u_1\partial_y$. It follows that 
\begin{equation}\label{main3pf3}
\begin{split}
&3\left<(\nabla_{w}\nabla_u R)(w,u)w,u\right>^2\\
&-2\left<(\nabla_w^2R)(w,u)w,u\right> \left<(\nabla_{u}^2R)(w,u)w,u\right> \\
&=16[(27-2a^2)u_2^4+(18-4a^2)u_1^2u_2^2+(27-2a^2)u_1^4].
\end{split}
\end{equation}

If $27-2a^2>0$, then we can set $u_1=0$ and see that the above expression is positive for some $u$. This contradicts with Theorem \ref{main2}. Therefore, we have $a\leq -\sqrt{\frac{27}{2}}$. Finally, we remark that the expression in (\ref{main3pf3}) is non-positive if $a\leq -\sqrt{\frac{27}{2}}$.
\end{proof}

\

\section{Proof of Theorem \ref{main2}}\label{proof}

In this section, we will give the proof of Theorem \ref{main2}. Let us first recall the notation that we are using. Let $\tau\mapsto U_{s,t}(\tau)$ be the parallel translation of the vector $u$ along the geodesic $\tau\mapsto \gamma_{s,t}(\tau):=\exp(\tau(tv+sw))$. Let $\tau\mapsto J_{s,t}(\tau)$ be the Jacobi field defined along the geodesic $\gamma_{s,t}$ which satisfies the conditions $J_{s,t}(0)=u$, $J_{s,t}(1)=0$, and $J_{s,t}(\tau)\neq 0$ for all $\tau$ in the interval $(0,1)$. 

First, it follows from (1) of Lemma \ref{longlem-1} and (\ref{long-1}) that 
\[
\begin{split}
&\partial_t\partial_s^2\left<R(\partial_\tau\gamma,J)\partial_\tau\gamma,U\right>\Big|_{s=t=0}
\\&=\left<R(\partial_s^2\partial_\tau\gamma,J)\partial_t\partial_\tau\gamma,U\right>
+\left<R(\partial_t\partial_\tau\gamma,J)\partial_s^2\partial_\tau\gamma,U\right>\\
&+2\left<(\partial_s R)(\partial_s\partial_\tau\gamma,J)\partial_t\partial_\tau\gamma,U\right> +2\left<(\partial_sR)(\partial_t\partial_\tau\gamma,J)\partial_s\partial_\tau\gamma,U\right>\\
&+2\left<R(\partial_s\partial_\tau\gamma,\partial_sJ)\partial_t\partial_\tau\gamma,U\right> +2\partial_t(\left<R(\partial_s\partial_\tau\gamma,J)\partial_s\partial_\tau\gamma,U\right>)\\
&+2\left<R(\partial_s\partial_\tau\gamma,J)\partial_t\partial_\tau\gamma,\partial_sU\right> +2\left<R(\partial_t\partial_\tau\gamma,\partial_sJ)\partial_s\partial_\tau\gamma,U\right>\\ &+2\left<R(\partial_t\partial_\tau\gamma,J)\partial_s\partial_\tau\gamma,\partial_sU\right>\Big|_{s=t=0}.
\end{split}
\]

By (2') of Lemma \ref{longlem-1}, (1) of Lemma \ref{longlem-2}, Lemma \ref{longlem-3}, (1) of Lemma \ref{longlem-4},  and (1) of Lemma \ref{longlem-6}, the above equation simplifies to 
\begin{equation}\label{main2-1}
\begin{split}
&\partial_t\partial_s^2\left<R(\partial_\tau\gamma,J)\partial_\tau\gamma,U\right>\Big|_{s=t=0}
\\&=2\tau(1-\tau)\left<(\nabla_wR)(w,u)v,u\right> +2\tau(1-\tau)\left<(\nabla_wR)(v,u)w,u\right>\\
& +2\partial_t(\left<R(\partial_s\partial_\tau\gamma,J)\partial_s\partial_\tau\gamma,U\right>)\Big|_{s=t=0}.
\end{split}
\end{equation}

By (1) of Lemma \ref{longlem-2}, (2) of Lemma \ref{longlem-4}, and (1) of Lemma \ref{longlem-6}, we also have 
\[
\partial_t(\left<R(\partial_s\partial_\tau\gamma,J)\partial_s\partial_\tau\gamma,U\right>)\Big|_{s=t=0} =\left<(\partial_tR)(\partial_s\partial_\tau\gamma,J)\partial_s\partial_\tau\gamma,U\right>\Big|_{s=t=0}.  
\]

Therefore, it follows from (2') of Lemma \ref{longlem-1} and Lemma \ref{longlem-3} that 
\[
\partial_t(\left<R(\partial_s\partial_\tau\gamma,J)\partial_s\partial_\tau\gamma,U\right>)\Big|_{s=t=0} =\tau(1-\tau)\left<(\nabla_vR)(w,u)w,u\right>.  
\]

If we combine this with (\ref{main2-1}), then we have 
\[
\begin{split}
&\partial_t\partial_s^2\left<R(\partial_\tau\gamma,J)\partial_\tau\gamma,U\right>\Big|_{s=t=0}
\\&=2\tau(1-\tau)\left<(\nabla_wR)(w,u)v,u\right> +2\tau(1-\tau)\left<(\nabla_wR)(v,u)w,u\right>\\
& +2\tau(1-\tau)\left<(\nabla_vR)(w,u)w,u\right>.
\end{split}
\]

Therefore, by Theorem \ref{main1}, we have 
\[
\begin{split}
&\partial_t\MTW(u,tv,w)\Big|_{t=0} \\
&=\frac{1}{4}\Big(\left<(\nabla_wR)(w,u)v,u\right> +\left<(\nabla_wR)(v,u)w,u\right>
\\ &+\left<(\nabla_vR)(w,u)w,u\right>\Big). 
\end{split}
\]

Finally, by taking covariant derivative of the property 
\[
\left<R(w,u)v,u\right>=\left<R(v,u)w,u\right>
\] 
of the Riemannian curvature $R$, we have 
\[
\begin{split}
&\partial_t\MTW(u,tv,w)\Big|_{t=0} \\
&=\frac{1}{2}\left<(\nabla_wR)(w,u)v,u\right>+\frac{1}{4}\left<(\nabla_vR)(w,u)w,u\right>. 
\end{split}
\]

By (1) of Lemma \ref{longlem-1}, (1) of Lemma \ref{longlem-2}, (1) and (2) of Lemma \ref{longlem-4}, (1) of Lemma \ref{longlem-6}, if we differentiate (\ref{long-1}) twice with respect to $t$, then 
\[
\begin{split}
&\partial_t^2\partial_s^2\left<R(\partial_\tau\gamma,J)\partial_\tau\gamma,U\right>\Big|_{t=s=0}\\
&=2\left<(\partial_s^2R)(\partial_t\partial_\tau\gamma,J)\partial_t\partial_\tau\gamma,U\right> +2\left<R(\partial_t\partial_s^2\partial_\tau\gamma,J)\partial_t\partial_\tau\gamma,U\right>\\
&+2\left<R(\partial_t\partial_\tau\gamma,\partial_s^2 J)\partial_t\partial_\tau\gamma,U\right> +2\left<R(\partial_t\partial_\tau\gamma,J)\partial_t\partial_s^2\partial_\tau\gamma,U\right>\\
&+2\left<R(\partial_t\partial_\tau\gamma,J)\partial_t\partial_\tau\gamma,\partial_s^2U\right> 
+4\left<(\partial_t\partial_s R)(\partial_s\partial_\tau\gamma,J)\partial_t\partial_\tau\gamma,U\right>\\
&+4\left<(\partial_sR)(\partial_t\partial_\tau\gamma,\partial_sJ)\partial_t\partial_\tau\gamma,U\right> +4\left<(\partial_t\partial_sR)(\partial_t\partial_\tau\gamma,J)\partial_s\partial_\tau\gamma,U\right>\\
&+4\left<(\partial_sR)(\partial_t\partial_\tau\gamma,J)\partial_t\partial_\tau\gamma,\partial_sU\right> +4\left<R(\partial_s\partial_\tau\gamma,\partial_t\partial_sJ)\partial_t\partial_\tau\gamma,U\right>\\
&+2\left<(\partial_t^2R)(\partial_s\partial_\tau\gamma,J)\partial_s\partial_\tau\gamma,U\right> 
+2\left<R(\partial_t^2\partial_s\partial_\tau\gamma,J)\partial_s\partial_\tau\gamma,U\right>\\
&+2\left<R(\partial_s\partial_\tau\gamma,\partial_t^2J)\partial_s\partial_\tau\gamma,U\right>
+2\left<R(\partial_s\partial_\tau\gamma,J)\partial_t^2\partial_s\partial_\tau\gamma,U\right>\\
&+4\left<R(\partial_s\partial_\tau\gamma,J)\partial_t\partial_\tau\gamma,\partial_t\partial_sU\right>
+4\left<R(\partial_t\partial_\tau\gamma,\partial_t\partial_sJ)\partial_s\partial_\tau\gamma,U\right>\\ &+4\left<R(\partial_t\partial_\tau\gamma,\partial_sJ)\partial_t\partial_\tau\gamma,\partial_sU\right> +4\left<R(\partial_t\partial_\tau\gamma,J)\partial_s\partial_\tau\gamma,\partial_t\partial_sU\right>\Big|_{s=t=0}.
\end{split}
\]

By (2') of Lemma \ref{longlem-1}, Lemma \ref{longlem-3}, (3) of Lemma \ref{longlem-4}, Lemma \ref{longlem-5}, (1) and (2)of Lemma \ref{longlem-6}, the above equation simplifies to 
\[
\begin{split}
&\partial_t^2\partial_s^2\left<R(\partial_\tau\gamma,J)\partial_\tau\gamma,U\right>\Big|_{t=s=0}\\
&=2(1-\tau)\tau^2\left<(\nabla_w^2R)(v,u)v,u\right> 
+2(1-\tau)\tau^2\left<R(R(v,w)w,u)v,u\right>\\
&+\frac{2\tau(\tau-1)(\tau-2)}{3}\left(\left<R(v,R(w,u)w)v,u\right>+\left<R(w,R(v,u)v)w,u\right>\right) \\
&+4(1-\tau)\tau^2\left<(\nabla_v\nabla_w R)(w,u)v,u\right> 
+4(1-\tau)\tau^2\left<(\nabla_v\nabla_wR)(v,u)w,u\right>\\
&+2(1-\tau)\tau^2\left<(\nabla_v^2R)(w,u)w,u\right> 
+2(1-\tau)\tau^2\left<R(R(v,w)v,u)w,u\right>\\
&+2(1-\tau)\tau^2\left<R(w,u)R(v,w)v,u\right> 
+2(1-\tau)\tau^2\left<R(v,u)R(v,w)w,u\right>\\
&+2(1-\tau)\tau^2\left<R(w,u)v,R(v,w)u\right> 
+2(1-\tau)\tau^2\left<R(v,u)w,R(v,w)u\right>\\
&+\frac{4\tau(\tau-1)(\tau-2)}{3}\left(\left<R(v,R(w,u)v)w,u\right>+\left<R(w,R(w,u)v)v,u\right>\right)\\
&-\frac{4\tau(\tau-1)(\tau+1)}{3}\left(\left<R(v,R(v,w)u)w,u\right>+\left<R(w,R(v,w)u)v,u\right>\right).
\end{split}
\]

We simplify the above equation further by using the property of the Riemannian curvature $R$. 
\[
\begin{split}
&\partial_t^2\partial_s^2\left<R(\partial_\tau\gamma,J)\partial_\tau\gamma,U\right>\Big|_{t=s=0}\\
&=2(1-\tau)\tau^2\left<(\nabla_w^2R)(v,u)v,u\right> 
-4(1-\tau)\tau^2\left<R(v,u)u,R(v,w)w\right>\\
&+\frac{4\tau(\tau-1)(\tau-2)}{3}\left<R(v,u)v,R(w,u)w\right>\\ 
&+8(1-\tau)\tau^2\left<(\nabla_v\nabla_w R)(w,u)v,u\right> \\
&+2(1-\tau)\tau^2\left<(\nabla_v^2R)(w,u)w,u\right> 
-4(1-\tau)\tau^2\left<R(w,u)u,R(v,w)v\right>\\
&+2(1-\tau)\tau^2\left<R(w,u)v,R(v,w)u\right> 
+2(1-\tau)\tau^2\left<R(v,u)w,R(v,w)u\right>\\
&+\frac{4\tau(\tau-1)(\tau-2)}{3}\left(\left<R(w,u)v,R(w,u)v\right>+\left<R(v,u)w,R(w,u)v\right>\right)\\
&-\frac{4\tau(\tau-1)(\tau+1)}{3}\left(\left<R(w,u)v,R(v,w)u\right>+\left<R(v,u)w,R(v,w)u\right>\right).
\end{split}
\]

If we integrate the above twice with respect to $\tau$ and multiply by $3/2$, then we have 
\[
\begin{split}
&\partial_t^2\MTW(u,tv,w)\Big|_{t=0}\\
&=\frac{1}{10}\left<(\nabla_w^2R)(v,u)v,u\right> 
-\frac{1}{5}\left<R(v,u)u,R(v,w)w\right>\\
&+\frac{4}{15}\left<R(v,u)v,R(w,u)w\right> 
+\frac{2}{5}\left<(\nabla_v\nabla_w R)(w,u)v,u\right> \\
&+\frac{1}{10}\left<(\nabla_v^2R)(w,u)w,u\right> 
-\frac{1}{5}\left<R(w,u)u,R(v,w)v\right>\\
&+\frac{4}{15}\left(\left<R(w,u)v,R(w,u)v\right>+\left<R(v,u)w,R(w,u)v\right>\right)\\
&+\frac{1}{3}\left(\left<R(w,u)v,R(v,w)u\right>+\left<R(v,u)w,R(v,w)u\right>\right).
\end{split}
\]

\

\section{Appendix 1: Lemmas for Natural Mechanical Actions}

In the two appendices, we give the proof of various lemmas used in the previous sections. The first appendix is devoted to those lemmas which are related to the natural mechanical actions. The rest of the lemmas needed only in the Riemannian case are done in the second appendix. Let us first recall our notations. Let $u$, $v$, and $w$ be tangent vectors based at a point $x$ which is a critical point of the potential $V$. Let $\tau\mapsto\gamma_{s,t}(\tau)$ be the curve of least action with initial velocity $tv+sw$. Let $\tau\mapsto U_{s,t}(\tau)$ be the parallel translation of the vector $u$ along the curve $\tau\mapsto \gamma_{s,t}(\tau)$. Let $\tau\mapsto J_{s,t}(\tau)$ be the Jacobi field defined along the curve $\tau\mapsto\gamma_{s,t}(\tau)$ by the Jacobi equation 
\[
\partial_\tau^2J+R(\partial+\tau\gamma,J)\partial_\tau\gamma+\Hess V(J)=0
\]
and satisfies the conditions $J_{s,t}(0)=u$, $J_{s,t}(1)=0$, and $J_{s,t}(\tau)\neq 0$ for all $\tau$ in the interval $(0,1)$. Let $\tau\mapsto \bar v(\tau)$ be the solution to the initial value problem 
\[
\partial_\tau^2 \bar v=-\Hess V_x(\bar v),\quad \bar v\Big|_{\tau=0}=0,\quad \partial_\tau\bar v\Big|_{\tau=0}=v. 
\]
Similarly, let $\bar w$ be the solution to the above initial value problem with $v$ replace by $w$. Let $\tilde u$ be the solution to the boundary value problem 
\[
\partial_\tau^2 \tilde u=-\Hess V_x(\tilde u),\quad \tilde u\Big|_{\tau=0}=u,\quad \tilde u\Big|_{\tau=1}=0, \quad \tilde u\neq 0 \quad \text{if } 0<\tau<1. 
\]

\begin{lem}\label{longlem-1}
The family of curves $\gamma$ satisfies the following:
\begin{enumerate}
\item $\partial_\tau\gamma\Big|_{s=t=0}=0$, 
\item $\partial_t\gamma\Big|_{s=t=0}=\bar v$, \quad $\partial_s\gamma\Big|_{s=t=0}=\bar w$, 
\end{enumerate}
In particular, if we are in the Riemannian case where the potential $V\equiv 0$, then we have 
\[
(2')\quad\quad\quad  \partial_t\gamma\Big|_{s=t=0}=\tau v, \quad \partial_s\gamma\Big|_{s=t=0}=\tau w.
\]
\end{lem}

\begin{proof}
Recall that $\gamma$ satisfies the Newton's equation $\partial_\tau^2\gamma=-\nabla V_\gamma$ with initial condition $\partial_\tau\gamma\Big|_{\tau=0}=tv+sw$. Since $x$ is a critical point of the potential $V$, it follows that $\gamma\Big|_{s=t=0}\equiv x$ is the solution to the above initial value problem with $s=t=0$. Therefore, (1) follows immediately from this. 

If we differentiate the Newton's equation with respect to $t$, then we have 
\[
R(\partial_t\gamma,\partial_\tau\gamma)\partial_\tau\gamma+\partial_\tau^2\partial_t\gamma =\partial_t\partial_\tau^2\gamma=-\Hess V_\gamma(\partial_t\gamma). 
\]
If we set $s=t=0$ and apply (1), we have 
\[
\partial_\tau^2\partial_t\gamma\Big|_{s=t=0} =-\Hess V_x(\partial_t\gamma\Big|_{s=t=0}). 
\]
We also have the initial conditions 
\[
\partial_t\gamma\Big|_{t=s=\tau=0}=0, \quad \partial_\tau\partial_t\gamma\Big|_{t=s=\tau=0}=v.
\] 
It follows that $\bar v=\partial_t\gamma\Big|_{s=t=0}$. 
\end{proof}

\begin{lem}\label{longlem-2}
The family $U$ of parallel vector fields satisfies
\begin{enumerate}
\item $\partial_sU\Big|_{s=t=0}=\partial_tU\Big|_{s=t=0}=0$,
\item $\partial_s^2U\Big|_{s=t=0}=\int_0^{\bar\tau} R(\partial_\tau\bar w,\bar w)ud\tau$, 
\item $\partial_t^2 U\Big|_{s=t=0}=\int_0^{\bar\tau} R(\partial_\tau\bar v,\bar v)ud\tau$. 
\end{enumerate}

In particular, $\partial_s^2U\Big|_{s=t=0}=\partial_t^2U\Big|_{s=t=0}=0$ in the Riemannian case where $V\equiv 0$. 
\end{lem}

\begin{proof}
Since the family is parallel, we have $\partial_\tau U=0$. Therefore, we have 
\[
0=\partial_s\partial_\tau U=R(\partial_s\gamma,\partial_\tau\gamma)U+\partial_\tau\partial_sU. 
\]
If we evaluate at $s=t=0$ and apply (1) of Lemma \ref{longlem-1}, then we have 
\[
\partial_\tau\partial_sU\Big|_{s=t=0}=0. 
\]
Therefore, $\partial_sU\Big|_{s=t=0}$ is constant in $\tau$ and we have 
\[
\partial_sU\Big|_{s=t=0}=\partial_sU\Big|_{\tau=s=t=0}=\partial_su=0. 
\]

For the proof of (2), we apply $\partial_\tau U=0$ again and get 
\[
\begin{split}
\partial_\tau\partial_t^2U\Big|_{s=t=0}
&=\partial_t\partial_\tau\partial_tU+R(\partial_\tau\gamma,\partial_t\gamma)\partial_tU\Big|_{s=t=0}\\
&=\partial_t^2\partial_\tau U +\partial_t(R(\partial_\tau\gamma,\partial_t\gamma)U) +R(\partial_\tau\gamma,\partial_t\gamma)\partial_tU \Big|_{s=t=0}\\
\end{split}
\]

By (1) and (2) of Lemma \ref{longlem-1}, the above equation becomes 
\[
\begin{split}
\partial_\tau\partial_t^2U\Big|_{s=t=0}
&=\partial_t^2\partial_\tau U +R(\partial_t\partial_\tau\gamma,\partial_t\gamma)U \Big|_{s=t=0}\\
&=R(\partial_\tau\bar v,\bar v)u. 
\end{split}
\]

If we integrate with respect to the $\tau$-variable, then 
\[
\begin{split}
\partial_t^2U\Big|_{s=t=0,\tau=\bar\tau}&=\partial_t^2U\Big|_{\tau=s=t=0} 
+\int_0^{\bar\tau} R(\partial_\tau\bar v,\bar v)ud\tau\\
&=\int_0^{\bar\tau} R(\partial_\tau\bar v,\bar v)ud\tau.
\end{split}
\]

By (2') of Lemma \ref{longlem-1}, $\bar v=\tau v$ in the Riemannian case. It follows from the skew symmetry of the Riemannian curvature that 
\[
\partial_t^2U\Big|_{s=t=0}=0. 
\]
\end{proof}

\begin{lem}\label{longlem-3}
The family of Jacobi fields $J$ satisfies 
\[
J\Big|_{s=t=0}=\tilde u
\]
In particular, if we are in the Riemannian case where the potential $V=0$, then we have 
\[
J\Big|_{s=t=0}=(1-\tau)u.
\]
\end{lem}

\begin{proof}
Recall that the Jacobi equation is 
\[
\partial_\tau^2 J+R(\partial_\tau\gamma,J)\partial_\tau\gamma+\Hess V_\gamma(J)=0. 
\]
If we set $s=t=0$ and apply (1) of Lemma \ref{longlem-1}, then the above equation becomes 
\[
\partial_\tau^2 J+\Hess V_x(J)\Big|_{s=t=0}=0. 
\]
We also have the boundary conditions $J\Big|_{\tau=0}=u$ and $J\Big|_{\tau=1}=0$. 
It follows that $J\Big|_{s=t=0}=\tilde u$. 

In the Riemannian case, $V\equiv 0$ and the above equation becomes $\partial_\tau^2 J\Big|_{s=t=0}=0$. If we combine this with the boundary conditions, we have $J\Big|_{s=t=0}=(1-\tau)u$. 
\end{proof}

\section{Appendix 2: Lemmas for the Riemannian case}

Let us first recall and specialize our notations used in the previous appendix to the Riemannian case. Let $u$, $v$, and $w$ be tangent vectors at a point $x$ and let $\tau\mapsto\gamma_{s,t}(\tau):=\exp(\tau(tv+sw))$ be the geodesic with initial velocity $tv+sw$. Let $\tau\mapsto U_{s,t}(\tau)$ be the parallel translation of the vector $u$ along the curve $\tau\mapsto \gamma_{s,t}(\tau)$. Let $\tau\mapsto J_{s,t}(\tau)$ be the Jacobi field defined along the curve $\tau\mapsto\gamma_{s,t}(\tau)$ which satisfies the conditions $J_{s,t}(0)=u$, $J_{s,t}(1)=0$, and $J_{s,t}(\tau)\neq 0$ for all $\tau$ in the interval $(0,1)$.

\begin{lem}\label{longlem-4}
The family of geodesics $\gamma$ satisfies the following:
\begin{enumerate}
\item $\partial_s^2\partial_\tau\gamma\Big|_{s=t=0}=\partial_t^2\partial_\tau\gamma\Big|_{s=t=0}=0$,
\item $\partial_t\partial_s\partial_\tau\gamma\Big|_{s=t=0}=\partial_s\partial_t\partial_\tau\gamma\Big|_{s=t=0}=0$. 
\item $\partial_t\partial_s^2\partial_\tau\gamma\Big|_{s=t=0}=\tau^2 R(v,w)w$, \quad $\partial_s\partial_t^2\partial_\tau\gamma\Big|_{s=t=0}=\tau^2 R(w,v)v$, 
\item $\partial_t^2\partial_s\partial_\tau\gamma\Big|_{s=t=0}=\tau^2 R(v,w)v$, \quad  $\partial_s^2\partial_t\partial_\tau\gamma\Big|_{s=t=0}=\tau^2 R(w,v)w$. 
\end{enumerate}
\end{lem}

\begin{proof}
For (1), we have 
\begin{equation}\label{longlem-4-1}
\partial_s^2\partial_\tau\gamma=\partial_\tau\partial_s^2\gamma+R(\partial_s\gamma,\partial_\tau\gamma)\partial_s\gamma. 
\end{equation}
Since $\gamma\Big|_{t=0}=\exp(s\tau w)$ is a geodesic in the variable $s$, we have 
\[
\partial_s^2\gamma\Big|_{s=t=0}=0.
\] 
Therefore, if we set $s=t=0$, then (\ref{longlem-4-1}) becomes 
\[
\partial_s^2\partial_\tau\gamma\Big|_{s=t=0}=R(\partial_s\gamma,\partial_\tau\gamma)\partial_s\gamma\Big|_{s=t=0}. 
\]
Finally, if we apply (1) of Lemma \ref{longlem-1}, then we obtain (1). 

For (2), we have 
\[
\begin{split}
&\partial_\tau\partial_t\partial_s\partial_\tau\gamma\Big|_{s=t=0}\\ &=R(\partial_\tau\gamma,\partial_t\gamma)\partial_s\partial_\tau\gamma +\partial_t\partial_\tau\partial_s\partial_\tau\gamma\Big|_{s=t=0}\\
&=R(\partial_\tau\gamma,\partial_t\gamma)\partial_s\partial_\tau\gamma +\partial_t(R(\partial_\tau\gamma,\partial_s\gamma)\partial_\tau\gamma) +\partial_t\partial_s\partial_\tau^2\gamma\Big|_{s=t=0}. 
\end{split}
\]
If we apply (1) of Lemma \ref{longlem-1}, then the above equation becomes 
\[
\partial_\tau\partial_t\partial_s\partial_\tau\gamma\Big|_{s=t=0} =\partial_t\partial_s\partial_\tau^2\gamma\Big|_{s=t=0}.
\]
Since $\gamma$ is a geodesic for each $t$ and $s$ (i.e. $\partial_\tau^2\gamma=0$), it follows that 
\[
\partial_\tau\partial_t\partial_s\partial_\tau\gamma\Big|_{s=t=0}=0.
\]

Therefore, $\partial_t\partial_s\partial_\tau\gamma\Big|_{s=t=0}$ is constant in $\tau$ and we have 
\[
\partial_t\partial_s\partial_\tau\gamma\Big|_{s=t=0}=\partial_t\partial_s\partial_\tau\gamma\Big|_{\tau=s=t=0}=\partial_t\partial_s(tv+sw)\Big|_{s=t=0}=0. 
\]

For (3), we have 
\[
\begin{split} 
&\partial_\tau\partial_t\partial_s^2\partial_\tau\gamma\Big|_{s=t=0} \\
&=R(\partial_\tau\gamma, \partial_t\gamma)\partial_s^2\partial_\tau\gamma +\partial_t\partial_\tau\partial_s^2\partial_\tau\gamma\Big|_{s=t=0}
\end{split}
\]

By (1) of Lemma \ref{longlem-1}, the above equation becomes 
\[
\partial_\tau\partial_t\partial_s^2\partial_\tau\gamma\Big|_{s=t=0} =\partial_t\partial_\tau\partial_s^2\partial_\tau\gamma\Big|_{s=t=0}.
\]

If we apply (1) of Lemma \ref{longlem-1} again to the above equation, then we obtain
\[
\begin{split} 
\partial_\tau\partial_t\partial_s^2\partial_\tau\gamma\Big|_{s=t=0} 
&=\partial_t\partial_s\partial_\tau\partial_s\partial_\tau\gamma
+ \partial_t(R(\partial_\tau\gamma,\partial_s\gamma)\partial_s\partial_\tau\gamma)\Big|_{s=t=0}\\
&=\partial_t\partial_s\partial_\tau\partial_s\partial_\tau\gamma
+ R(\partial_t\partial_\tau\gamma,\partial_s\gamma)\partial_s\partial_\tau\gamma\Big|_{s=t=0}. 
\end{split}
\]

By (2') of Lemma \ref{longlem-1}, we have 
\[
\partial_\tau\partial_t\partial_s^2\partial_\tau\gamma\Big|_{s=t=0} 
=\partial_t\partial_s\partial_\tau\partial_s\partial_\tau\gamma
+ \tau R(v,w)w\Big|_{s=t=0}
\]

Since $\tau\mapsto\gamma$ is a geodesic, the above equation becomes 
\[
\begin{split} 
\partial_\tau\partial_t\partial_s^2\partial_\tau\gamma\Big|_{s=t=0} 
&=\partial_t\partial_s(R(\partial_\tau\gamma,\partial_s\gamma)\partial_\tau\gamma)
+ \tau R(v,w)w\Big|_{s=t=0}.\\
\end{split}
\]

If we apply (1) and (2') of Lemma \ref{longlem-1}, then we have 
\[
\begin{split} 
&\partial_\tau\partial_t\partial_s^2\partial_\tau\gamma\Big|_{s=t=0} \\
&=R(\partial_t\partial_\tau\gamma,\partial_s\gamma)\partial_s\partial_\tau\gamma +R(\partial_s\partial_\tau\gamma,\partial_s\gamma)\partial_t\partial_\tau\gamma
+ \tau R(v,w)w\Big|_{s=t=0}\\
&=2\tau R(v,w)w. 
\end{split}
\]

If we integrate the above equation in $\tau$, then we get 
\[
\begin{split}
\partial_t\partial_s^2\partial_\tau\gamma\Big|_{s=t=0} 
&=\tau^2 R(v,w)w+\partial_t\partial_s^2\partial_\tau\gamma\Big|_{\tau=s=t=0} \\
&=\tau^2 R(v,w)w+\partial_t\partial_s^2(tv+sw)\Big|_{\tau=s=t=0} \\
&=\tau^2 R(v,w)w. 
\end{split}
\]
This finishes the proof of (3). 

For (4), we first apply (1) of Lemma \ref{longlem-1}. 
\[
\begin{split} 
\partial_\tau\partial_t^2\partial_s\partial_\tau\gamma\Big|_{s=t=0} 
&=R(\partial_\tau\gamma, \partial_t\gamma)\partial_t\partial_s\partial_\tau\gamma +\partial_t\partial_\tau\partial_t\partial_s\partial_\tau\gamma\Big|_{s=t=0} \\ 
&=\partial_t\partial_\tau\partial_t\partial_s\partial_\tau\gamma\Big|_{s=t=0}. 
\end{split}
\]

By (1) of Lemma \ref{longlem-1} again, the above equation becomes 
\[
\begin{split} 
\partial_\tau\partial_t^2\partial_s\partial_\tau\gamma\Big|_{s=t=0} 
&=\partial_t^2\partial_\tau\partial_s\partial_\tau\gamma
+ \partial_t(R(\partial_\tau\gamma,\partial_t\gamma)\partial_s\partial_\tau\gamma)\Big|_{s=t=0}\\
&=\partial_t^2\partial_\tau\partial_s\partial_\tau\gamma
+ R(\partial_t\partial_\tau\gamma,\partial_t\gamma)\partial_s\partial_\tau\gamma\Big|_{s=t=0}. 
\end{split}
\]

By (2') of Lemma \ref{longlem-1}, we have 
\[
\begin{split} 
\partial_\tau\partial_t^2\partial_s\partial_\tau\gamma\Big|_{s=t=0} 
&=\partial_t^2\partial_\tau\partial_s\partial_\tau\gamma
+ \tau R(v,v)w\Big|_{s=t=0}\\
&=\partial_t^2\partial_\tau\partial_s\partial_\tau\gamma\Big|_{s=t=0}.
\end{split}
\]

Since $\tau\mapsto\gamma$ is a geodesic, we get 
\[
\begin{split} 
\partial_\tau\partial_t^2\partial_s\partial_\tau\gamma\Big|_{s=t=0} 
&=\partial_t^2\partial_\tau\partial_\tau^2\gamma +\partial_t^2(R(\partial_\tau\gamma,\partial_s\gamma)\partial_\tau\gamma)\Big|_{s=t=0}\\
&=\partial_t^2(R(\partial_\tau\gamma,\partial_s\gamma)\partial_\tau\gamma)\Big|_{s=t=0}.
\end{split}
\]

By (1) and (2') of Lemma \ref{longlem-1}, the above equation becomes 
\[
\begin{split} 
\partial_\tau\partial_t^2\partial_s\partial_\tau\gamma\Big|_{s=t=0} 
&=2R(\partial_t\partial_\tau\gamma,\partial_s\gamma)\partial_t\partial_\tau\gamma\Big|_{s=t=0}\\
&=2\tau R(v,w)v\Big|_{s=t=0}.
\end{split}
\]

Finally, if we integrate the above equation in $\tau$, then we obtain 
\[
\begin{split}
\partial_t^2\partial_s\partial_\tau\gamma\Big|_{s=t=0} 
&=\tau^2 R(v,w)v+\partial_t^2\partial_s\partial_\tau\gamma\Big|_{\tau=s=t=0} \\
&=\tau^2 R(v,w)v+\partial_t^2\partial_s(tv+sw)\Big|_{s=t=0} \\
&=\tau^2 R(v,w)v. 
\end{split}
\]

\end{proof}

\begin{lem}\label{longlem-5}
The family of parallel vector fields $U$ satisfies 
\[
\partial_t\partial_sU\Big|_{s=t=0}=\frac{\tau^2}{2} R(v,w)u, \quad \partial_s\partial_tU\Big|_{s=t=0}=\frac{\tau^2}{2} R(w,v)u.
\]
\end{lem}

\begin{proof}
By (1) of Lemma \ref{longlem-1}, we have 
\[
\begin{split}
\partial_\tau\partial_t\partial_s U\Big|_{s=t=0}
&=\partial_t\partial_\tau\partial_sU+R(\partial_\tau\gamma,\partial_t\gamma)\partial_sU\Big|_{s=t=0}\\
&=\partial_t\partial_\tau\partial_sU\Big|_{s=t=0}. 
\end{split}
\]

If we apply (1) of Lemma \ref{longlem-1} again, then the above becomes 
\[
\begin{split}
\partial_\tau\partial_t\partial_s U\Big|_{s=t=0}
&=\partial_t\partial_s\partial_\tau U+\partial_t(R(\partial_\tau\gamma,\partial_s\gamma)U)\Big|_{s=t=0}\\
&=\partial_t\partial_s\partial_\tau U+R(\partial_t\partial_\tau\gamma,\partial_s\gamma)U\Big|_{s=t=0}. 
\end{split}
\]

Since $\tau\mapsto U$ is a parallel vector field, we get 
\[
\begin{split}
\partial_\tau\partial_t\partial_s U\Big|_{s=t=0}
&=R(\partial_t\partial_\tau\gamma,\partial_s\gamma)U\Big|_{s=t=0}.
\end{split}
\]

If we apply (1) and (2') of Lemma \ref{longlem-1}, then we have 
\[
\partial_\tau\partial_t\partial_s U\Big|_{s=t=0}=\tau R(v,w)u. 
\]

Since $U\Big|_{\tau=0}=u$, we can integrate the above equation in $\tau$ and obtain 
\[
\partial_t\partial_s U\Big|_{s=t=0}=\frac{\tau^2}{2}R(v,w)u+\partial_t\partial_s U\Big|_{\tau=s=t=0} =\frac{\tau^2}{2}R(v,w)u.
\]
\end{proof}

\begin{lem}\label{longlem-6}
The family of Jacobi fields $J$ satisfies 
\begin{enumerate}
\item $\partial_s J\Big|_{s=t=0}=\partial_t J\Big|_{s=t=0}=0$,
\item $\partial_t^2J\Big|_{s=t=0}=\frac{\tau(\tau-1)(\tau-2)}{3}R(v,u)v$, 
\item $\partial_s^2J\Big|_{s=t=0}=\frac{\tau(\tau-1)(\tau-2)}{3}R(w,u)w$, 
\item $\partial_t\partial_sJ\Big|_{s=t=0}=\frac{\tau(\tau-1)}{3} [(\tau-2)R(w,u)v-(\tau+1)R(v,w)u]$. 
\end{enumerate}
\end{lem}

\begin{proof}
The family $J$ satisfies the Jacobi equation 
\[
\partial_\tau^2J+R(\partial_\tau\gamma,J)\partial_\tau\gamma=0
\]
and the boundary conditions $J\Big|_{\tau=0}=u$ and $J\Big|_{\tau=1}=0$. 

If we differentiate the Jacobi equation with respect to $s$ and apply (1) of Lemma \ref{longlem-1}, then we have 
\[
\partial_s\partial_\tau^2J =-\partial_s(R(\partial_\tau\gamma,J)\partial_\tau\gamma)=0. 
\]

It follows that 
\[
\begin{split}
0&=\partial_s\partial_\tau^2J\Big|_{s=t=0}\\
&=\partial_\tau\partial_s\partial_\tau J +R(\partial_s\gamma,\partial_\tau\gamma)\partial_\tau J\Big|_{s=t=0}\\
&=\partial_\tau^2\partial_s J +R(\partial_s\gamma,\partial_\tau\gamma)\partial_\tau J +\partial_\tau(R(\partial_s\gamma,\partial_\tau\gamma)J)\Big|_{s=t=0}. 
\end{split}
\]

By (1) of Lemma \ref{longlem-1}, we have 
\[
\partial_\tau^2\partial_sJ\Big|_{s=t=0}=0. 
\]

This together with the boundary conditions 
$\partial_sJ\Big|_{\tau=0}=\partial_sJ\Big|_{\tau=1}=0$ give 
$\partial_s J\Big|_{s=t=0}=0$. This finishes the proof of (1). 

For (2), we differentiate the Jacobi equation with respect to $t$ twice and apply Lemma \ref{longlem-1}. 
\[
\begin{split}
\partial_t^2\partial_\tau^2J\Big|_{s=t=0}
&=-\partial_t^2(R(\partial_\tau\gamma,J)\partial_\tau\gamma)\Big|_{s=t=0}\\
&=-2R(\partial_t\partial_\tau\gamma,J)\partial_t\partial_\tau\gamma\Big|_{s=t=0}. 
\end{split}
\]
Therefore, by (2') of Lemma \ref{longlem-1} and Lemma \ref{longlem-3}, we have 
\begin{equation}\label{longlem-6-1}
\partial_t^2\partial_\tau^2J\Big|_{s=t=0}=-2(1-\tau)R(v,u)v. 
\end{equation}

On the other hand, by (1) of Lemma \ref{longlem-1}, we have 
\[
\begin{split}
\partial_t^2\partial_\tau^2 J\Big|_{s=t=0}
&=\partial_t\partial_\tau\partial_t\partial_\tau J
+\partial_t (R(\partial_t\gamma,\partial_\tau\gamma)\partial_\tau J)\Big|_{s=t=0}\\
&=\partial_t\partial_\tau\partial_t\partial_\tau J 
+R(\partial_t\gamma,\partial_t \partial_\tau\gamma)\partial_\tau J\Big|_{s=t=0}. 
\end{split}
\]

By (2') of Lemma \ref{longlem-1}, the above equation becomes 
\[
\begin{split}
\partial_t^2\partial_\tau^2 J\Big|_{s=t=0}
&=\partial_t\partial_\tau\partial_t\partial_\tau J 
+\tau R(v,v)\partial_\tau J\Big|_{s=t=0}\\
&=\partial_t\partial_\tau\partial_t\partial_\tau J\Big|_{s=t=0} \\
&=\partial_\tau\partial_t^2\partial_\tau J +R(\partial_t\gamma,\partial_\tau\gamma)\partial_t\partial_\tau J\Big|_{s=t=0}. 
\end{split}
\]

If we apply (1) of Lemma \ref{longlem-1}, then we have 
\[
\begin{split}
\partial_t^2\partial_\tau^2 J\Big|_{s=t=0}
&=\partial_\tau\partial_t^2\partial_\tau J\Big|_{s=t=0}\\
&=\partial_\tau\partial_t\partial_\tau\partial_t J 
+\partial_\tau\partial_t(R(\partial_t\gamma,\partial_\tau\gamma)J)\Big|_{s=t=0}\\ &=\partial_\tau\partial_t\partial_\tau\partial_t J \Big|_{s=t=0} +\partial_\tau(R(\partial_t\gamma,\partial_t\partial_\tau\gamma)J)\Big|_{s=t=0}. 
\end{split}
\]

By (2') of Lemma \ref{longlem-1}, we have 
\[
\begin{split}
\partial_t^2\partial_\tau^2 J\Big|_{s=t=0}
&=\partial_\tau\partial_t\partial_\tau\partial_t J \Big|_{s=t=0} +\partial_\tau(R(\tau v,v)J)\Big|_{s=t=0}\\
&=\partial_\tau\partial_t\partial_\tau\partial_t J \Big|_{s=t=0}. 
\end{split}
\]

By applying (1) of Lemma \ref{longlem-1}, the above equation becomes 
\[
\begin{split}
\partial_t^2\partial_\tau^2 J\Big|_{s=t=0}
&=\partial_\tau^2\partial_t^2 J +\partial_\tau(R(\partial_t\gamma,\partial_\tau\gamma)\partial_tJ)\Big|_{s=t=0}\\ &=\partial_\tau^2\partial_t^2 J\Big|_{s=t=0}. 
\end{split}
\]

If we combine this with (\ref{longlem-6-1}), then we have 
\[
\partial_\tau^2\partial_t^2 J=-2(1-\tau)R(v,u)v. 
\]

This together with the boundary conditions $\partial_t^2 J\Big|_{\tau=0}=\partial_t^2 J\Big|_{\tau=1}=0$ gives us
\[
\partial_t^2 J=\frac{\tau(\tau-1)(\tau-2)}{3} R(v,u)v. 
\]
This finishes the proof of (2). The proof of (3) is the same with $v$ and $w$ interchange. 

For (4), we apply (1) of Lemma \ref{longlem-1} and get 
\[
\begin{split}
&\partial_\tau^2\partial_t\partial_sJ\Big|_{s=t=0}\\
&=\partial_\tau\partial_t\partial_\tau\partial_sJ +\partial_\tau(R(\partial_\tau\gamma,\partial_t\gamma)\partial_sJ)\Big|_{s=t=0}\\
&=\partial_\tau\partial_t\partial_\tau\partial_sJ\Big|_{s=t=0}\\
\end{split}
\]

If we apply again (1) of Lemma \ref{longlem-1}, we get 
\[
\begin{split}
&\partial_\tau^2\partial_t\partial_sJ\Big|_{s=t=0}\\
&=\partial_t\partial_\tau^2\partial_sJ +R(\partial_\tau\gamma,\partial_t\gamma)\partial_\tau\partial_sJ\Big|_{s=t=0}\\
&=\partial_t\partial_\tau^2\partial_sJ\Big|_{s=t=0}\\
&=\partial_t\partial_\tau\partial_s\partial_\tau J +\partial_t\partial_\tau(R(\partial_\tau\gamma,\partial_s\gamma)J)\Big|_{s=t=0}.\\
\end{split}
\]

Using the fact that $\tau\mapsto \gamma$ is a geodesic and applying (1) of Lemma \ref{longlem-1}, we get 
\[
\begin{split}
&\partial_\tau^2\partial_t\partial_sJ\Big|_{s=t=0}\\
&=\partial_t\partial_\tau\partial_s\partial_\tau J +R(\partial_t\partial_\tau\gamma,\partial_\tau\partial_s\gamma)J) +R(\partial_t\partial_\tau\gamma,\partial_s\gamma)\partial_\tau J\Big|_{s=t=0}. 
\end{split}
\]

If we apply (2') of Lemma \ref{longlem-1} and Lemma \ref{longlem-3}, then we have
\[
\begin{split}
&\partial_\tau^2\partial_t\partial_sJ\Big|_{s=t=0}\\
&=\partial_t\partial_\tau\partial_s\partial_\tau J \Big|_{s=t=0}+(1-2\tau)R(v,w)u\\
&=\partial_t\partial_s\partial_\tau^2 J +\partial_t (R(\partial_\tau\gamma,\partial_s\gamma)\partial_\tau J)\Big|_{s=t=0}+(1-2\tau)R(v,w)u. 
\end{split}
\]

By (1), (2') of Lemma \ref{longlem-1}, and Lemma \ref{longlem-3}, we get 
\[
\begin{split}
&\partial_\tau^2\partial_t\partial_sJ\Big|_{s=t=0}\\
&=\partial_t\partial_s\partial_\tau^2 J\Big|_{s=t=0}+(1-3\tau)R(v,w)u. 
\end{split}
\]

If we apply the Jacobi equation and (1) of Lemma \ref{longlem-1}, then 
\[
\begin{split}
&\partial_\tau^2\partial_t\partial_sJ\Big|_{s=t=0}\\
&=-\partial_t\partial_s(R(\partial_\tau\gamma,J)\partial_\tau\gamma)\Big|_{s=t=0}+(1-3\tau)R(v,w)u\\
&=-R(\partial_s\partial_\tau\gamma,J)\partial_t\partial_\tau\gamma) -R(\partial_t\partial_\tau\gamma,J)\partial_s\partial_\tau\gamma)\Big|_{s=t=0}\\ 
&+(1-3\tau)R(v,w)u\\
\end{split}
\]

By (2') of Lemma \ref{longlem-1} and Lemma \ref{longlem-3}, we have 
\[
\begin{split}
&\partial_\tau^2\partial_t\partial_sJ\Big|_{s=t=0}\\
&=-(1-\tau)(R(w,u)v+R(v,u)w)+(1-3\tau)R(v,w)u.
\end{split}
\]

Finally, by the first Bianchi identity, we have 
\[
\partial_\tau^2\partial_t\partial_sJ\Big|_{s=t=0}=-2(1-\tau)R(w,u)v-2\tau R(v,w)u. 
\]
Using this and the boundary conditions $\partial_t\partial_sJ\Big|_{\tau=0}=\partial_t\partial_sJ\Big|_{\tau=1}=0$, we have 
\[
\partial_t\partial_s J=\frac{\tau(\tau-1)}{3} [(\tau-2)R(w,u)v-(\tau+1)R(v,w)u]. 
\]

\end{proof}

\section*{Acknowledgment}
The author would like to thank Professor Andrei Agrachev, Robert Bryant, and Robert McCann for their helpful suggestions and fruitful discussions.

\end{document}